\documentclass{amsart}
\usepackage{graphicx}
\usepackage{pslatex}
\usepackage{amsmath,amsfonts}
\usepackage{amstext}
\usepackage{amsthm}	
\usepackage{rotating}
\usepackage{multicol}
\usepackage{nicefrac}
\usepackage{amssymb}
\usepackage{verbatim}
\usepackage{rotating}
\usepackage{setspace}
\usepackage{mathrsfs}
\usepackage{dsfont}
\usepackage{epstopdf}
\usepackage[dvipsnames]{xcolor}
\usepackage{fancyhdr}
\usepackage{textcomp}
\usepackage{indentfirst}
\usepackage{latexsym,verbatim,pdfsync,color}

\newtheorem{thm}{Theorem}[section]
\newtheorem{cor}[thm]{Corollary}
\newtheorem{prop}[thm]{Proposition}
\newtheorem{lem}[thm]{Lemma}
\newtheorem*{thm*}{Theorem}

\theoremstyle{definition}
\newtheorem{defn}[thm]{Definition}
\newtheorem{ese}[thm]{Example}
\newtheorem{rmk}[thm]{Remark}

\newtheorem*{nota}{Notation}
\newtheorem*{guess}{Guess}
\newtheorem*{rmk*}{Remark}

\newcommand{\dsR}{\mathds{R}}
\newcommand{\dsZ}{\mathds{Z}}
\newcommand{\dsC}{\mathds{C}}

\newcommand{\clC}{\mathcal{C}}
\newcommand{\clD}{\mathcal{D}}
\newcommand{\clE}{\mathcal{E}}

\newcommand{\clL}{\mathcal{L}}
\newcommand{\clO}{\mathcal{O}}

\newcommand{\clR}{\mathcal{R}}

\newcommand{\clW}{\mathcal{W}}

\newcommand{\bthm}{\begin{thm}}
\newcommand{\bproof}{\begin{proof}}
\newcommand{\eproof}{\end{proof}}
\newcommand{\ethm}{\end{thm}}
\newcommand{\bdefn}{\begin{defn}}
\newcommand{\edefn}{\end{defn}}
\newcommand{\bprop}{\begin{prop}}
\newcommand{\eprop}{\end{prop}}
\newcommand{\bcor}{\begin{cor}}
\newcommand{\ecor}{\end{cor}}                   
\newcommand{\blem}{\begin{lem}}
\newcommand{\elem}{\end{lem}}
\newcommand{\bese}{\begin{ese}}
\newcommand{\eese}{\end{ese}}
\newcommand{\boss}{\begin{oss}}                 
\newcommand{\eoss}{\end{oss}}
\newcommand{\bnota}{\begin{nota}}
\newcommand{\enota}{\end{nota}}
\newcommand{\bguess}{\begin{guess}}
\newcommand{\eguess}{\end{guess}}
\newcommand{\bequa}{\begin{equation}}
\newcommand{\eequa}{\end{equation}}
\newcommand{\brmk}{\begin{rmk}}
\newcommand{\ermk}{\end{rmk}}

\newcommand\quant{\advance\quantno by1
                     
\ifnum\quantno=1\qquad\else\quad\fi\forall }

\DeclareMathOperator{\sgn}{sgn}

\DeclareMathOperator{\ch}{Ch}

\newcommand{\lt}{\left(}
\newcommand{\rt}{\right)}
\newcommand{\lquad}{\left[}
\newcommand{\rquad}{\right]}
\newcommand{\lgra}{\left\{}
\newcommand{\rgra}{\right\}}

\def\im{\operatorname{Im}}
\def\re{\operatorname{Re}}

\newcommand{\pih}{\frac{\pi}{2}}

\def\p{\partial }
\def\II{I\!I}
\def\III{I\!I\!I}

\def\ER{ \color{black} }

 \title[A comparison between the Bergman and Szeg\H{o} kernels of $D'_\beta$ ]{A comparison between the Bergman and\\ Szeg\H{o} kernels of the non-smooth worm domain $D'_\beta$}
 \author{Alessandro Monguzzi}
 \date{\today}

\numberwithin{equation}{section}

\begin{document}

\begin{abstract}
In this work we provide an asymptotic expansion for the Szeg\H{o} kernel associated to a suitably defined Hardy space on the non-smooth worm domain $D'_\beta$. After describing the singularities of the kernel, we compare it with an asymptotic expansion of the Bergman kernel. In particular, we show that the Bergman kernel has the same singularities of the first derivative of the Szeg\H{o} kernel with respect to any of the variables. On the side, we prove the boundedness of the Bergman projection operator on Sobolev spaces of integer order.
\end{abstract}

\address{Dipartimento di Matematica\\ 
Universit\`{a} Statale di Milano\\ Via C. Saldini 
50, 20133 Milan\\
Italy}
\email{alessandro.monguzzi@unimi.it}

\thanks{The author is partially supported by the grant PRIN 2010-11 {\em Real and Complex Manifolds: Geometry, Topology and Harmonic Analysis} of the Italian Ministry of Education (MIUR)}

\maketitle
\section{Introduction and Main Results}
A classical problem in complex analysis is the study of the Szeg\H{o} projection operator associated to a domain $\Omega\subseteq\dsC^n$. If $\rho$ is a defining function for $\Omega$, i.e., $\Omega=\{z\in\dsC^n:\rho(z)<0\}$ and $\nabla \rho\neq0$ on the topological boundary $b\Omega$, the Hardy space $H^2(\Omega)$ is classically defined as 
$$
H^2(\Omega)=\left\{F\ \text{holomorphic in}\ \Omega: \|F\|^2_{H^2(\Omega)} =\sup_{\varepsilon>0}\int_{b\Omega_\varepsilon}|F(\zeta)|^2\ d\sigma_{\varepsilon}<\infty \right\}
$$
where $\Omega_\varepsilon=\{z\in\dsC^n:\rho(z)<-\varepsilon\}$ and $d\sigma_\varepsilon$ is the Euclidean measure induced on $b\Omega_\varepsilon$. 
Every function $F$ in $H^2(\Omega)$ admits a boundary value function $\widetilde{F}$ and the linear space of these boundary value functions defines a closed subspace of $L^2(b\Omega)$ which we denote by $H^2(b\Omega)$. The Szeg\H{o} projection operator is the orthogonal projection operator 
$$
\widetilde{S_{\Omega}}: L^2(b\Omega)\to H^2(b\Omega).
$$

The operator $\widetilde{S_\Omega}$ has an integral representation by means of the Szeg\H{o} kernel $K_{\Omega}$; we refer to \cite{MR0473215} for more details. The geometry of the domain $\Omega$ is reflected in the kernel $K_{\Omega}$, hence on the mapping properties of $\widetilde{S_\Omega}$. The behavior of the Szeg\H{o} projection has been extensively studied in the last 40 years in many different settings; among others, we refer the reader to the papers \cite{MR0450623, MR773403, MR835396, MR871667, MR971689, MR999739, MR979602, MR1133741, MR1452048, MR1094488, MR1381988, MR906810, MR1452048, MR2030575, MR3084008, MR3145917, 2015arXiv150400287M} and the references therein. 

The smooth worm domain $\clW=\clW_\beta$ does not
belong to any of the known situations. The domain $\clW$ was first introduced by Diederich and Forn\ae ss in \cite{MR0430315} as a counterexample to certain classical conjectures about the geometry of pseudoconvex domains; for instance, the domain $\clW$ is an example of a smooth bounded pseudoconvex domain with nontrivial \emph{Nebenh\"{u}lle}.

For $\beta>\pih$ , the worm is the domain
\begin{equation}\label{SmoothWorm}
\clW=\{(z_1,z_2)\in\dsC^2:|z_1-e^{i\log|z_2|^2}|^2<1-\eta(\log|z_2|^2), z_2\neq 0\},
\end{equation}
where $\eta$ is a smooth, even, convex, non-negative function on the real line, chosen so that $\eta^{-1}(0)=[-\beta+\pih,\beta-\pih]$ and so that $\clW$ is bounded, smooth and weakly pseudoconvex. We refer to the survey paper \cite{MR2393268} for a history of the study of the worm domain and related problems.

 Due to the peculiarity of the worm domain and the lack of general results regarding the regularity of the Szeg\H{o} projection of smooth bounded weakly pseudoconvex domains,  the study of the operator $\widetilde{S_{\clW}}$ is a natural and interesting question.
 
 In order to obtain information about $\widetilde{S_{\clW}}$, it is useful to study the same problem for a simpler model of $\clW$, that is, the non--smooth worm domain $D'_\beta$. 
  Recently the author studied in \cite{2015arXiv150400287M} the $L^p$ and Sobolev mapping properties of the Szeg\H{o} projection operator associated to $D'_\beta$, namely,
$$
D'_\beta=\lgra (z_1,z_2)\in \dsC^2 : \left|\im z_1-\log|z_2|^2 \right|<\frac{\pi}{2}, \left|\log|z_2|^2\right|<\beta-\frac{\pi}{2}\rgra.
$$

\begin{figure}[h]
\begin{center}
\includegraphics[width=10cm]{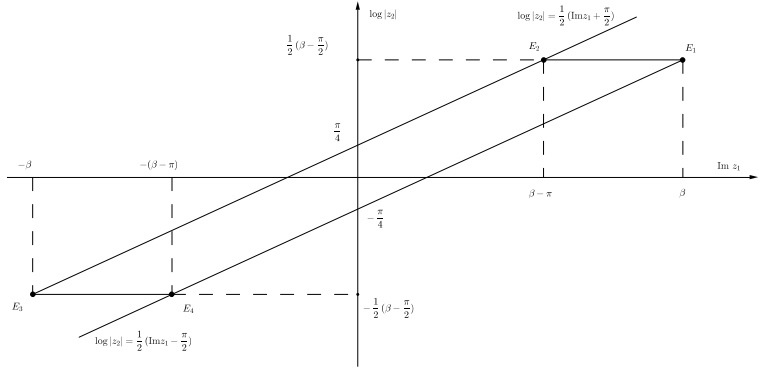}
\end{center}
\caption{A representation of the domain $D'_\beta$ in the $(\im z_1,\log|z_2|)$-plane.}\label{WormStrip}
\end{figure}

The domain $D'_\beta$ is a simpler model of $\clW$ which has been already used to study the mapping properties of the Bergman projection operator associated to $\clW$, that is, the Hilbert space projection operator from $L^2(\clW)$ onto the closed subspace of holomorphic functions. For results concerning  the Bergman projection and kernel of the worm domain, the role of the domain $D'_\beta$ and related results we refer the reader to \cite{MR1149863, MR2904008, MR1370592, MR1128596,2014arXiv1410.8490K,  MR2448387, 2015arXiv150400287M} and the references therein.

  In \cite{2015arXiv150400287M} the Hardy space $H^2(D'_\beta)$ is defined as the function space
  $$
  H^2(D'_\beta)=\Big\{F\ \text{holomorphic in }\ D'_\beta: \|F\|^2_{H^2(D'_\beta)}=\sup_{(t,s)\in [0,\pih)\times[0,\beta-\pih)}\clL_2F(t,s)<\infty\Big\},
  $$
where 
\begin{align*}
&\clL_2F (t,s)=\\
&\int\limits_{\dsR}\int\limits_{0}^{1}\left|F\lt x+i(s+t),e^{\frac{s}{2}}e^{2\pi i\theta} \rt\right|^2 d\theta dx +\int\limits_{\dsR}\int\limits_{0}^{1}\left|F\lt x-i(s+t),e^{-\frac{s}{2}}e^{2\pi i\theta} \rt\right|^2 d\theta dx\\
& +\int\limits_{\dsR}\int\limits_{0}^{1}\left|F\lt x+i(s-t),e^{\frac{s}{2}}e^{2\pi i\theta} \rt\right|^2\ d\theta dx\ +\int\limits_{\dsR}\int\limits_{0}^{1}\left|F\lt x-i(s-t),e^{-\frac{s}{2}}e^{2\pi i\theta} \rt\right|^2 d\theta dx. 
\end{align*}
Therefore, the space $H^2(D'_\beta)$ is defined considering a growth condition not on the topological boundary $bD'_\beta$, but on the distinguished boundary $\p D'_\beta$.  In detail, the distinguished boundary $\partial D'_\beta$ is the set  
\begin{align*}
\partial D'_\beta&=E_1\cup E_2\cup E_3\cup E_4,
\end{align*}
where
\begin{align*}
&E_1=\lgra(z_1,z_2): \im z_1=\beta,\log|z_2|^2=\beta-\pih\rgra;\\
&E_2=\lgra(z_1,z_2): \im z_1=\beta-\pi,\log|z_2|^2=\beta-\pih\rgra; \quad \\
&E_3=\lgra(z_1,z_2)\: \im z_1=-\beta,\log|z_2|^2=-\lt\beta-\pih\rt\rgra;\\
&E_4=\lgra(z_1,z_2): \im z_1=-(\beta-\pi),\log|z_2|^2=-\lt\beta-\pih\rt\rgra.
\end{align*}

Adapting a decomposition introduced by Barrett in \cite{MR1149863} and using some classical results for the Hardy spaces of a strip,  the following result is proved in \cite{2015arXiv150400287M}.
\begin{thm}[\cite{2015arXiv150400287M}]
The Hardy space $H^2(D'_\beta)$ is a reproducing kernel Hilbert space with respect to the inner product
$$
\big<F,G\big>_{H^2(D'_\beta)}:=\big<\widetilde F,\widetilde G\big>_{L^2(\p D'_\beta)}
$$
where $\widetilde{F}$ and $\widetilde{G}$ are the boundary value functions of $F$ and $G$ respectively. 
Moreover, the reproducing kernel of $H^2(D'_\beta)$, i.e., the Szeg\H{o} kernel, is given by
\begin{align}\label{KernelSum} \nonumber
K_{D'_\beta}[(w_1,w_2),(z_1,z_2)]&=\sum_{j\in\dsZ}{w_2}^j \overline{z_2}^jk_{j}(w_1,z_2)\\
&=\sum\limits_{j\in\dsZ}\frac {{w_2}^j \overline{z_2}^j}{8\pi}\int_{\dsR}\frac{e^{i(w_1-\overline{z_1 })\xi}}{\ch[\pi\xi]\ch[(2\beta-\pi)(\xi-\frac{j}{2})]}\ d\xi, 
\end{align}
where the series converges in $H^2(D'_\beta)$ for every fixed $(z_1,z_2)$ in $D'_\beta$.
\end{thm}
 
 We point out that we use the notation $\ch(x)$ instead of $\cosh(x)$ to denote the hyperbolic cosine of $x$. 

Unlike the Bergman case studied by Krantz and Peloso in \cite{MR2448387}, the boundedness results in \cite{2015arXiv150400287M} are proved without relying on an asymptotic expansion of the sum in \eqref{KernelSum}. Nevertheless, following \cite{MR2448387}, we provide in this work and asymptotic expansion of the Szeg\H{o} kernel $K_{D'_\beta}$ in order to enrich and complete the study of the Hardy spaces of $D'_\beta$ begun in \cite{2015arXiv150400287M}. In detail, we prove the following result.
\begin{thm}\label{t:kernel}
Let be $\beta>\pi$ and define $\nu_\beta=\frac{\pi}{2\beta-\pi}$. Let $h$ be fixed such that $\frac{\nu_\beta}{2}<h<\min\left(\frac{1}{2},\frac{3\nu_\beta}{2}\right).$ Then, there exist functions $\rho_1,\rho_2, G_1,\ldots, G_8, E$ and $\widetilde{E}$ that are holomorphic in $w$ and anti-holomorphic in $z$, for $w=(w_1,w_2)$ and $z=(z_1,z_2)$ varying in a neighborhood of $D'_\beta$, and remain bounded, together with all their derivatives, for $w,z\in \overline{D'_\beta}$, as $|\re(w_1-\overline{z_1})|\to +\infty$ such that
\begin{align}\label{expansion}
 K&_{D'_\beta}(w,z)=e^{-\sgn[\re(w_1-\overline{z_1})]\frac{(w_1-\overline{z_1})\nu_\beta}{2}}K(w,z)+e^{-\sgn[\re(w_1-\overline{z_1})](w_1-\overline{z_1}) h}\widetilde{K}(w,z)
\end{align}
where
\begin{align*}
  K(w,z)&=\frac{\rho_1(w,z)}{e^{\frac{\pi-i(w_1-\overline{z_1})}{2}}-w_2\overline{z_2}}+\frac{\rho_2(w,z)}{e^{-\frac{i(w_1-\overline{z_1})+\pi}{2}}-w_2\overline{z_2} }+E(w,z)\\
  &:=K_1(w,z)+K_2(w,z)+E(w,z)
 \end{align*} 
and
 \begin{align*}
  \widetilde{K}(w,z)&=\frac{G_1(w,z)}{[i(w_1-\overline{z_1})+2\beta][e^{\beta-\pih}-w_2\overline{z_2}]}+\frac{G_2(w,z)}{[e^{-\frac{i(w_1-\overline{z_1})+\pi}{2}}-w_2\overline{z_2}][i(w_1-\overline{z_1})+2\beta]}\\
  &\   +\frac{G_3(w,z)}{[e^{\frac{\pi-i(w_1-\overline{z_1})}{2}}-w_2\overline{z_2}][e^{\beta-\pih}-w_2\overline{z_2}]}+\frac{G_4(w,z)}{[e^{\frac{\pi-i(w_1-\overline{z_1})}{2}}-w_2\overline{z_2}][i(w_1-\overline{z_1})-2\beta]} \\
  &\  +\frac{G_5(w,z)}{[i(w_1-\overline{z_1})-2\beta][e^{-(\beta-\pih)}-w_2\overline{z_2}]}+\frac{G_6(w,z)}{[e^{-\frac{i(w_1-\overline{z_1})+\pi}{2}}-w_2\overline{z_2}][e^{-(\beta-\pih)}-w_2\overline{z_2}]}\\
  &\ +\frac{G_7(w,z)}{e^{\frac{\pi-i(w_1-\overline{z_1})}{2}}-w_2\overline{z_2}}\!+\!\frac{G_8(w,z)}{e^{-\frac{i(w_1-\overline{z_1})+\pi}{2}}-w_2\overline{z_2}}+\widetilde{E}(w,z)\\ \smallskip
  &:=\widetilde{K_1}(w,z)+\ldots+\widetilde{K_8}(w,z)+\widetilde{E}(w,z).
\end{align*}
\end{thm}
Notice that the definition of $D'_\beta$ (as well as the one of $\clW$) requires only that $\beta>\pih$. For simplicity of the arguments, we restrict ourselves to the case $\beta>\pi$. This is not a serious constraint since the most interesting situations for the worm domain occur when $\beta$ tends to $+\infty$. 

After proving the expansion \eqref{expansion}, we compare it with the expansion of the Bergman kernel contained in \cite{MR2448387}. The investigation of the relationship between the Szeg\H{o} and the Bergman kernel is a natural problem, but, to the best of the author's knowledge, not much is known about it. Stein posed the problem of investigating this relationship in \cite[pg. $20$]{MR0473215} and made some comments about some specific situations such as the case of the unit ball where the Szeg\H{o} and Bergman kernel are explicitly known. In \cite{MR979602} the authors study the Bergman and Szeg\H{o} kernels of pseudoconvex domains of finite type in $\dsC^2$. Moreover, they observe that in the special case of model domains of the form
$$
\Omega=\{(z_1,z_2)\in\dsC^2:\im z_2> P(z_1)\}
$$ 
where $P$ is a subharmonic, nonharmonic polynomial on $\dsC$, the Bergman kernel can be expressed as a derivative of the Szeg\H{o} kernel. More recently, Hirachi \cite{MR2087043} realized the asymptotic expansion of the Bergman and Szeg\H{o} kernels of strictly pseudoconvex domains as special values of a family of meromorphic functions and Chen--Fu studied in \cite{MR2836124} the ratio of the Szeg\H{o} and Bergman kernels for smooth bounded pseudoconvex domains in $\dsC^n$. Finally, Krantz shows in \cite{MR3160814} how to connect the Bergman and Szeg\H{o} kernels of strongly pseudoconvex domains via Stoke's theorem.

Here we remark a direct connection between the asymptotic expansion \eqref{expansion} of the Szeg\H{o} kernel and the asymptotic expansion of the Bergman kernel $B_{D'_\beta}$. In detail, we show that the derivative of the expansion \eqref{expansion} with respect to any of the complex variables $w_j,\overline{z_j}, j=1,2$ has the same types of singularities of the Bergman kernel. Therefore, we have a link between the Szeg\H{o} and Bergman kernels similar to the one observed in \cite{MR979602} for model domains. 

Let us recall the asymptotic expansion of the Bergman kernel proved by Krantz and Peloso. The theorem we state here is slightly different from the one stated in \cite{MR2448387}, but it is not hard to deduce it.
\begin{thm}[\cite{MR2448387}]\label{KP}
Let be $\beta>\pi$ and define $\nu_\beta=\frac{\pi}{2\beta-\pi}$ . Let $h$ be fixed such that $\nu_\beta<h<\min(1,2\nu_\beta)$. Then, there exist functions $\varphi_1, \varphi_2, F_1, \ldots, F_8$ that are holomorphic in $w$ and anti-holomorphic in $z$, for $w=(w_1, w_2)$ and $z=(z_1, z_2)$ varying in a neighborhood of $D'_\beta$, and having size $\clO(|\re w_1-\re z_1|)$,  together with all their derivatives, for $w,z\in\overline{D'_\beta}$, as $|\re(w_1-\overline{z_1})|\to+\infty$. Moreover, there exist functions $E,\widetilde{E}\in \clC^{\infty}(\overline{D'_\beta}\times\overline{D'_\beta})$ such that 
$$
D^{\alpha}_{w_1}D^\gamma_{z_1}E(w,z), D^{\alpha}_{w_1}D^\gamma_{z_1}\widetilde{E}(w,z)=\clO(|\re w_1-\re z_1|^{|\alpha|+|\gamma|}).
$$
as $|\re w_1-\re z_1|\to+\infty$. Then, the following holds. 
\bequa\label{ExpansionBergman}
B_{D'_\beta}(w,z)=e^{-\sgn[\re(w_1-\overline{z_1})](w_1-\overline{z_1})\nu_\beta }B(w,z)+ e^{-\sgn[\re(w_1-\overline{z_1})](w_1-\overline{z_1})h}\widetilde{B}(w,z).
\eequa 
The functions $B(w,z)$ and $\widetilde{B}(w,z)$ are given by

\begin{align*}
  B(w,z)&=\frac{\varphi_1(w,z)}{[e^{\frac{\pi-i(w_1-\overline{z_1})}{2}}-w_2\overline{z_2}]^2}+\frac{\varphi_2(w,z)}{[e^{-\frac{\pi+i(w_1-\overline{z_1})}{2}}-w_2\overline{z_2}]^2}+E(w,z)\\
  &:= B_1(w,z)+B_2(w,z)+E(w,z)
\end{align*}
and
\begin{align*}
  \widetilde{B}(w,z)=&\frac{F_1(w,z)}{[i(w_1-\overline{z_1})+2\beta]^2[e^{\beta-\pih}-w_2\overline{z_2}]^2}\\
  &+\frac{F_2(w,z)}{[i(w_1-\overline{z_1})+2\beta]^2[w_2\overline{z_2}-e^{-\frac{i(w_1-\overline{z_1})+\pi}{2}}]^2}\\
  &+\frac{F_3(w,z)}{[e^{\frac{\pi-i(w_1-\overline{z_1})}{2}}-w_z\overline{z_2}]^2[e^{\beta-\pih}-w_2\overline{z_2}]^2}\\
&+\frac{F_4(w,z)}{[i(w_1-\overline{z_1})-2\beta]^2[e^{\frac{\pi-i(w_1-\overline{z_1})}{2}}-w_2\overline{z_2}]^2}\\
 &+\frac{F_5(w,z)}{[i(w_1-\overline{z_1})-2\beta]^2[w_2\overline{z_2}-e^{-(\beta-\pih)}]^2}\\
 &+\frac{F_6(w,z)}{[e^{-\frac{i(w_1-\overline{z_1})+\pi}{2}}-w_2\overline{z_2}]^2[w_2\overline{z_2}-e^{-(\beta-\pih)}]^2}\\
 &+\frac{F_7(w,z)}{[i(w_1-\overline{z_1})+2\beta]^2[w_2\overline{z_2}-e^{\beta-\pih}][e^{-\frac{i(w_1-\overline{z_1})+\pi}{2}}-w_2\overline{z_2}]}\\
 &+\frac{F_8(w,z)}{[i(w_1-\overline{z_1})-2\beta]^2[w_2\overline{z_2}-e^{-(\beta-\pih)}][e^{\frac{\pi+i(w_1-\overline{z_1})}{2}}-w_2\overline{z_2}]}+\widetilde{E}(w,z)\\
 &:=\widetilde{B}_1(w,z)+\ldots+\widetilde{B}_8(w,z)+\widetilde{E}(w,z).
\end{align*}
\end{thm}
Then, we prove the following result.

\begin{thm}\label{t:comparison}
Let $(\tilde{w},\tilde{z})$ be a point in the topological boundary $bD'_\beta$ of the non-smooth worm domain $D'_\beta$. Then,
\[
\lim_{(w,z)\to(\tilde w,\tilde z)}\left|\frac{\p}{\p w_1}K_{D'_\beta}(w,z)\Big/B_{D'_\beta}(w,z)\right|=C(\tilde w,\tilde z)
\]
where $C(\tilde w,\tilde z)$ is a positive constant depending only on the point $(\tilde w,\tilde z)$. The same conclusion holds if we consider the derivative of $K_{D'_\beta}$ with respect to any of the variables $w_2, \overline{z_1}$ or $\overline z_2$.
\end{thm}
Finally, we conclude the paper with a result about the mapping properties of the Bergman projection operator $P_{D'_\beta}$. Let $A^p(D'_\beta)$ denotes the Bergman space, that is the space of functions in $L^p$ and holomorphic on $D'_\beta$. In \cite{MR2448387} the authors study the $L^p$ mapping properties of $P_{D'_\beta}$ and, using the expansion \eqref{ExpansionBergman}, they prove the following theorem. 
\begin{thm}[\cite{MR2448387}]\label{KP2} The Bergman projection operator $P_{D'_\beta}$ extends to a bounded linear operator $P_{D'_\beta}:L^p(D'_\beta)\to A^p(D'_\beta)$ for every $p$ in $(1,+\infty)$.
\end{thm}
Here we use again \eqref{ExpansionBergman} to prove the regularity of the Bergman projection $P_{D'_\beta}$ in Sobolev scale.
\begin{thm}\label{t:Sobolev}
Let $k$ be a positive integer. Then, the Bergman projection operator extends to a bounded linear operator
$$
P_{D'_\beta}: W^{k,p}(D'_\beta)\to W^{k,p}(D'_\beta)
$$
for every $p\in(1,+\infty)$.
\end{thm}

We conclude the introduction with a remark on Theorem \ref{t:kernel} and on a difference between the Szeg\H{o} and the Bergman setting.

Given two biholomorphic domains $D_1$ and $D_2$, it is well- known  how to write the Bergman projection $P_{D_1}$ in term of $P_{D_2}$ and vice versa. This is a general result guaranteed by the transformation rule of the Bergman kernel under biholomorphic mappings (see, for instance, \cite{MR1846625}). Hence, from \eqref{ExpansionBergman},  Krantz and Peloso also obtain an asymptotic expansion for another important non-smooth version of the worm domain. Namely, let $D_\beta$ be the domain
$$
 D_{\beta}=\left\{(z_1,z_2)\in\dsC^2: \re\big(z_1 e^{-i\log|z_2|^2}\big)>0, \big|\log|z_2|^2\big|<\beta-\pih\right\}.
$$
Then, the domains $D'_\beta$ and $D_\beta$ are biholomorphically equivalent via the map $\varphi: D'_\beta\to D_\beta$, $(z_1,z_2)\mapsto (e^{z_1},z_2)$.
Both $D'_\beta$ and $D_\beta$ have a central role in the study of the Bergman projection attached the smooth worm $\clW$. We refer the reader to \cite{MR1149863, MR2393268} for further details.

In the Szeg\H{o} setting we lack a general transformation rule for the Szeg\H{o} kernel under biholomorphic mappings, therefore we cannot trivially use the asymptotic expansion of $K_{D'_\beta}$ in Theorem \ref{t:kernel} in order to obtain information on the Szeg\H{o} kernel attached to $D_\beta$. Thus, the Szeg\H{o} projection of the domain $D_\beta$ must be independently studied. 

The paper is organized as follows. In Section \ref{singularities} we describe the singularities of $K_{D'_\beta}$, whereas in Section \ref{kernel} we prove Theorem \ref{t:kernel}. The proof of the theorem is quite long, therefore we proceed step by step proving a series of different propositions and lemmas. In Section \ref{comparison} we compare the asymptotic expansion of the Szeg\H{o} kernel with the asymptotic expansion of the Bergman kernel and in Section \ref{SobolevBergman} we prove Theorem \ref{t:Sobolev}.

\section{The singularities of $K_{D'_\beta}$}\label{singularities}
In this section we describe the behavior of $K_{D'_\beta}$ at the boundary. In particular we observe that, even if the space $H^2(D'_\beta)$ is defined relying on the distinguished boundary $\partial D'_\beta$, the kernel $K_{D'_\beta}$ is singular on the whole topological boundary $bD'_\beta$. The privileged role of the distinguished boundary is echoed in the fact that $K_{D'_\beta}$ has the worst behavior on $\p D'_\beta\times \p D'_\beta$. 

Using the notation of Theorem \ref{t:kernel}, we notice the following facts:
\begin{itemize}

\item[$-$] for $w,z\in D'_\beta$ the terms $K_1$ and $\widetilde{K}_1$ become singular only if
$$
w_2\overline{z_2}\to e^{-\frac{i(w_1-\overline{z_1})+\pi}{2}}.
$$ 
This can happen only if $\log|w_2|^2\to \im(w_1)-\pih$ and $\log|z_2|^2\to \im(z_1)-\pih$. Thus, $K_1$ and $\widetilde{K}_1$ are singular only when both  $w$ and $z$ tend to the right oblique boundary line of the domain in Figure \ref{WormStrip};
\item[$-$]  the terms $K_2$ and $\widetilde{K}_2$ are similar to $K_1$ and $\widetilde{K}_1$ and they are singular on the left oblique boundary line of the domain in Figure \ref{WormStrip};
\item[$-$] the term $\widetilde{K}_3$ is singular when 
 \begin{align*}
&w_2\overline{z}_2\to e^{-\frac{i(w_1-\overline{z_1})+\pi}{2}}\qquad
\text{or}  \qquad w_2\overline{z}_2\to e^{-(\beta-\pih)}.
\end{align*} \ER
Thus, $\widetilde{K}_3$ is singular when both $w$ and $z$ tend either to the lower horizontal  or the right oblique boundary line on of the domain in Figure \ref{WormStrip}. Notice that the worst behavior of the term is  when $w_2\overline{z}_2\to e^{-(\beta-\pih)}$ and $(w_1-\overline{z_1})\to 2(\beta-\pi)$ since the singularities add up. Therefore, $\widetilde{K}_3$ has the worst behavior on the component $E_4$ of $\partial D'_\beta$;
\item[$-$]  the term $\widetilde{K}_4$ is singular when
\begin{align*}
&w_2\overline{z}_2\to e^{-\frac{i(w_1-\overline{z_1})+\pi}{2}}\qquad
\text{or} \qquad \im(w_1-\overline{z_1})\to2\beta.
\end{align*} \ER
Therefore, $\widetilde{K}_4$ is singular when both $w,z$ tend to the right oblique boundary line of the domain of Figure \ref{WormStrip} and it has the worst behavior on $E_1$; 
\item[$-$] the singularities of $\widetilde{K}_5$ are similar to the ones of $\widetilde{K}_4$ and the worst situation is when both $w,z$ tend to $E_3$;
\item[$-$] the singularities of $\widetilde{K_6}$ are similar to the ones of $\widetilde{K}_3$. The term in singular both $w$ and $z$ tend to left oblique or the upper boundary line of the domain in Figure \ref{WormStrip}and the worst situation is when both $w,z$ tend to$E_2$;
\item[$-$]the term $\widetilde{K}_7$ becomes singular when
\begin{align*}
&w_2\overline{z}_2\to e^{\beta-\pih}\qquad
\text{or} \qquad \im(w_1-\overline{z_1})\to2\beta.
\end{align*} \ER
Therefore, the term becomes singular when both $w$ and $z$ tend to the upper boundary line of the domain in Figure \ref{WormStrip} and, like $\widetilde{K}_4$, it has the worst behavior when both $w,z$ tend to $E_1$;
\item[$-$] the last term $\widetilde{K}_8$ is symmetric to $\widetilde{K}_7$. It is singular when $w,z$ tends to the lower boundary line of the domain in Figure \ref{WormStrip} and it has the worst behavior when both $w,z$ tend to $E_3$ .
\end{itemize}

\section{The reproducing kernel of $H^2(D'_\beta)$}\label{kernel}
This section is devoted to the proof of Theorem \ref{t:kernel}. The proof is based on a direct computation of the sum \eqref{KernelSum}. In order to simplify the notation, we define
$$
I_j(\tau)=\int_\dsR\frac{e^{i\tau\xi}}{\ch[\pi\xi]\ch[(2\beta-\pi)(\xi-\frac{j}{2})]}\ d\xi.
$$
Then, we would like to compute the sum
\begin{equation}\label{sum-Ij}
\sum_{j\in\dsZ}I_j(\tau)\lambda^j,
\end{equation}
where the couple $(\tau,\lambda)$ belongs to the set
$$
\clD=\{(\tau,\lambda)\in\dsC^2:\big|\im\tau-\log|\lambda|^2\big|<\pi,\, e^{-(\beta-\pih)}<|\lambda|<e^{\beta-\pih}\}.
$$
Similarly to \cite{MR2448387}, in order to compute \eqref{sum-Ij}, we compute $I_j(\tau)$ by means of the residue theorem and we keep track of the error terms that arise. 

Let us denote by $g_j$ the holomorphic function $g_j:\dsC\to \dsC$
$$
g_j(\zeta):=\frac{e^{i\tau\zeta}}{\ch[\pi\zeta]\ch[(2\beta-\pi)(\zeta-\frac{j}{2})]}.
$$
About the function $g_j$, we have the following result whose easy proof we do not include.
\bprop\label{p:residues}
The function $g_j$ is holomorphic in the plane except at the points

$$
\zeta=i\lt\frac{1}{2}+k\rt,\ k\in\dsZ, \ \ \ \ \ \zeta=i\nu_\beta\lt\ \frac{1}{2}+k\rt+\frac{j}{2},\ k\in\dsZ, 
$$
where $\nu_\beta=\frac{\pi}{2\beta-\pi}$. Moreover

$$
Res\lt g_j,\frac{j}{2}\pm i\frac{\nu_\beta}{2}\rt=\pm\frac{e^{i\tau(\frac{j}{2}\pm i\frac{\nu_\beta}{2})}}{i(2\beta-\pi)\ch\lquad\pi\lt\frac{j}{2}\pm i\frac{\nu_\beta}{2}\rt\rquad}.
$$
\eprop
To compute $I_j(\tau)$ we shall distinguish two cases according to whether $\re\tau\geq0$ or $\re\tau<0$. Let us focus now on the case $\re\tau\geq 0$. We shall use the method of contour integrals. As contour of integration we choose the rectangular box $\gamma_N$ centered on the imaginary axis with corners $N+i0$, $-N+i0$, $N+ih$ and $N-ih$ where $h$ is chosen so that
\begin{equation}\label{h}
\frac{\nu_\beta}{2}<h<\min\lt\frac{1}{2},\frac{3\nu_\beta}{2}\rt.
\end{equation}
We have the following result.

\bprop
Let $\beta>\pi$ and fix $h$ as above. We define

$$
R_j(\tau)=2\pi i \cdot Res\lt g_j, \frac{j}{2}+i\frac{\nu_\beta}{2}\rt,\ \ \ \ \ J_j(\tau)=\int_{\dsR}g_j(\xi+ih)\ d\xi.
$$
Then, for all $j$ in $\dsZ$, 

$$
I_j(\tau)=R_j(\tau)+J_j(\tau).
$$
\eprop
\begin{proof}
The residue theorem guarantees that
\begin{align*}
  \int_{-N}^{N}g_j(\xi)d\xi=R_j(\tau)\!+\!\int_{-N}^{N}g_j(\xi+ih) d\xi-i\int_{0}^h& g_j(N+i\xi) d\xi-i\int_{h}^{0}g_j(-N+i\xi) d\xi.
\end{align*}
It is not hard to prove that the integrals along the vertical sides go to zero and obtain the conclusion.

The proof for $\re\tau<0$ is completely analogous, but we integrate along the similar rectangular box in the bottom half-plane.
\end{proof}

Thus, we have split the sum (\ref{sum-Ij}) into two different sums, namely

$$
\sum_{j\in\dsZ} I_j(\tau)\lambda^j=\sum_{j\in\dsZ} R_j(\tau)\lambda^j+\sum_{j\in\dsZ}J_j(\tau)\lambda^j.
$$
\begin{rmk}\label{positive}
For simplicity of notation, from now on, we restrict ourselves to work with $\re\tau\geq0$. The case $\re \tau<0$ is similarly obtained.
\end{rmk}
\begin{rmk}
 The following equality will have a prominent role in our computations. Let $a,b$  in $\dsR$ such that $a\neq 0$, then
\begin{equation}\label{fond-equality}
\frac{e^{|a|}}{\ch(a+ib)}=2e^{-i\sgn(a)b}\lt1-\frac{e^{-2\sgn(a)(a+ib)}}{1+e^{-2\sgn(a)(a+ib)}}\rt.
\end{equation}
\end{rmk}
\subsection{The sum of the $R_j$'s}
We prove the following proposition.
\bprop
There exists a function $E(\tau,\lambda)$ which is smooth in a neighborhood of $\overline{\clD}$ such that $D^\alpha_{\tau} E=\mathcal{O}(|\re \tau|^\alpha)$ as $|\re\tau|\to+\infty$ and such that
\begin{align}\label{sum-Rj}
\clR(\tau,\lambda)
& =\sum_{j\in\dsZ}R_j(\tau)\lambda^j =\frac{4\nu_\beta}{e^{\frac{\tau\nu_\beta}{2}}}\Bigg\{
\frac{e^{\frac{i\pi\nu_\beta}{2}}}{\lambda
  e^{\frac{i\tau+\pi}{2}}-1}+\frac{e^{-\frac{i\pi\nu_\beta}{2}}\lambda
  e^{\frac{i\tau-\pi}{2}}}{1-\lambda
  e^{\frac{i\tau-\pi}{2}}}+\frac{1}{\ch[i\frac{\nu_\beta}{2}]}+E(\tau,\lambda)\Bigg\} .
\end{align}
 The convergence of the series is uniform on compact subsets of $\clD$.
\eprop
\begin{proof}
From Proposition \ref{p:residues} we have
\begin{align*}
R_j(\tau)&=2\pi i\, \text{Res}\lt
g_j,\frac{j}{2}+i\frac{\nu_\beta}{2}\rt=\frac{2\pi
  e^{i\tau\lt\frac{j}{2}+i\frac{\nu_\beta}{2}\rt}}{(2\beta-\pi)\ch\lquad\pi\lt\frac{j}{2}+i\frac{\nu_\beta}{2}\rt\rquad}=\frac{2\nu_\beta e^{i\tau\lt\frac{j}{2}+i\frac{\nu_\beta}{2}\rt}}{\ch\lquad\pi\lt\frac{j}{2}+i\frac{\nu_\beta}{2}\rt\rquad}.
\end{align*}
Our problem is to compute the sum
\begin{align}\label{sum-Rj-1}
\sum_{j\in\dsZ}\frac{2\nu_\beta
  e^{i\tau\lt\frac{j}{2}+i\frac{\nu_\beta}{2}\rt}}{\ch\lquad\pi\lt\frac{j}{2}+i\frac{\nu_\beta}{2}\rt\rquad}\lambda^j=2\nu_\beta
e^{-\frac{\tau\nu_\beta}{2}}\sum_{j\in\dsZ}\frac{e^{\frac{ij\tau}{2}}\lambda^j}{\ch\lquad\pi\lt\frac{j}{2}+i\frac{\nu_\beta}{2}\rt\rquad} .
\end{align}
If we consider only the sum on the right-hand side of the previous equation, from (\ref{fond-equality}), it follows
\begin{align*}
\sum_{j\in\dsZ}\frac{e^{\frac{ij\tau}{2}}\lambda^j}{\ch\lquad\pi\lt\frac{j}{2}+i\frac{\nu_\beta}{2}\rt\rquad}&=2\sum_{j\in\dsZ}e^{\frac{ij\tau}{2}}\lambda^j\frac{e^{-i\sigma(j)\frac{\pi\nu_\beta}{2}}}{e^{\frac{|j|\pi}{2}}}\!\!\lquad 1-\frac{e^{-2\sigma(j)\lt \frac{j\pi}{2}+i\frac{\pi\nu_\beta}{2}\rt}}{1+e^{-2\sigma(j)\lt \frac{j\pi}{2}+i\frac{\pi\nu_\beta}{2}\rt}}\rquad\\
&=2\lt F-E+G\rt,
\end{align*}
where
\begin{align*}
&\sigma(j)=\sgn(j);\qquad\qquad F=F(\tau,\lambda)=\sum\limits_{j\neq 0}e^{\frac{ij\tau}{2}}\lambda^j e^{-\frac{\pi}{2}\lt |j|+i\nu_\beta \sigma(j)\rt};\\
&G=\frac{1}{\ch\lt i\frac{\pi\nu_\beta}{2}\rt;}\qquad \qquad E=E(\tau,\lambda)= \sum\limits_{j\neq 0}\frac{e^{\frac{ij\tau}{2}}\lambda^j e^{-\frac{\pi}{2}\lt |j|+i\nu_\beta \sigma(j)\rt} e^{-\pi\sigma(j)(j+i\nu_\beta)}}{1+e^{-\pi\sigma(j)(j+i\nu_\beta)}}.
\end{align*}
About $F$, we have
\begin{align}\label{F1}\nonumber
F&=e^{\frac{i\pi\nu_\beta}{2}}\sum_{j<0} e^{j\lt\frac{i\tau}{2}+\frac{\pi}{2}\rt}\lambda^j+e^{-\frac{i\pi\nu_\beta}{2}}\sum_{j>0}e^{j\lt\frac{i\tau}{2}-\frac{\pi}{2}\rt}\lambda^j=\frac{e^{\frac{i\pi\nu_\beta}{2}}}{\lambda e^{\frac{i\tau+\pi}{2}}-1}+\frac{e^{-\frac{i\pi\nu_\beta}{2}}\lambda e^{\frac{i\tau-\pi}{2}}}{1-\lambda e^{\frac{i\tau-\pi}{2}}}
\end{align}
and the convergence of the two series is guaranteed exactly when
$e^{\frac{\textrm{Im}\tau-\pi}{2}}<|\lambda|<e^{\frac{\textrm{Im}\tau+\pi}{2}}$. 

We analyze now the error term $E$. It results
\begin{align*}
E&= e^{\frac{3i\pi\nu_\beta}{2}}\sum_{j<0}\frac{\lambda^j e^{j\frac{i\tau+3\pi}{2}}}{1+e^{\pi(j+i\nu_\beta)}}+e^{-\frac{3i\pi\nu_\beta}{2}}\sum_{j>0}\frac{\lambda^j e^{j\frac{i\tau-3\pi}{2}}}{1+e^{-\pi(j+i\nu_\beta)}}. 
\end{align*}
It is easy to prove that there exists a constant $c>0$ such that $|1+e^{-\pi\sigma(j)+i\nu_\beta}|>c>0$ for every $j$.
Hence the series which define $E$ converge when
$e^{\frac{\textrm{Im}\tau-3\pi}{2}}<|\lambda|<e^{\frac{\textrm{Im}\tau+3\pi}{2}}$
which is an annulus strictly larger than
$e^{\frac{\textrm{Im}\tau-\pi}{2}}<|\lambda|<e^{\frac{\textrm{Im}\tau+\pi}{2}}$. Thus the sums of the two series are smooth and bounded, with all derivatives smooth and bounded, in a neighborhood of $\overline{\clD}$. Moreover, since the two series converge uniformly on a neighborhood of $\overline{\clD}$, we can differentiate term by term with respect to $\tau$ and easily obtain that $D^{\alpha}_\tau E=\clO (|\re\tau|^\alpha)$ as $|\re \tau|\to\infty$. The conclusion follows.
\end{proof}
\subsection{The sum of the $J_j$'s}
It remains to compute $\sum J_j(\tau)\lambda^j$. 
We recall that
$$
J_j(\tau)=\int_{\dsR}\frac{e^{i\tau(\xi+ih)}}{\ch[\pi(\xi+ih)]\ch[(2\beta-\pi)(\xi+ih-j/2)]}\ d\xi.
$$

If we define $\sigma(\xi)=e^{-i\sgn(\xi)\pi h-i\sgn(\xi-\frac{j}{2})(2\beta-\pi)h}$, from equation (\ref{fond-equality}) we obtain 
\begin{equation}\label{decomposition-Jj}
J_j(\tau)=4e^{-\tau h}\lt M_j(\tau)-E^{(1)}_j(\tau)-E^{(2)}_j(\tau)+E^{(3)}_j(\tau)\rt,
\end{equation}
where
\begin{align}
&M_j(\tau)= \int_{\dsR}\sigma(\xi)\frac{e^{i\tau\xi}}{e^{\pi|\xi|+(2\beta-\pi)|\xi-\frac{j}{2}|}}\ d\xi;\label{Mj}\\
&E^{(1)}_j(\tau)=\int_{\dsR}\sigma(\xi)\frac{e^{i\tau\xi}}{e^{\pi|\xi|+(2\beta-\pi)|\xi-\frac{j}{2}|}}\lquad\frac{e^{-2\sgn(\xi)[\pi(\xi+ih)]}}{1+e^{-2\sgn(\xi)[\pi(\xi+ih)]}}\rquad\ d\xi;\label{E1j}\\
&E^{(2)}_j(\tau)=\int_\dsR\sigma(\xi)\frac{e^{i\tau\xi}}{e^{\pi|\xi|+(2\beta-\pi)|\xi-\frac{j}{2}|}}\lquad\frac{e^{-2\sgn(\xi-\frac{j}{2})[(2\beta-\pi)(\xi-\frac{j}{2}+ih)]}}{1\!+e^{-2\sgn(\xi-\frac{j}{2})[(2\beta-\pi)(\xi-\frac{j}{2}+ih)]}}\rquad\ d\xi;\label{E2j}\\  \nonumber
&E^{(3)}_j(\tau)=\int_{\dsR}\sigma(\xi)\frac{e^{i\tau\xi}}{e^{\pi|\xi|+(2\beta-\pi)|\xi-\frac{j}{2}|}}\\
&\hspace{27mm}\times\lquad\frac{e^{-2\sgn(\xi)[\pi(\xi+ih)]}}{1+e^{-2\sgn(\xi)[\pi(\xi+ih)]}}\rquad\lquad\frac{e^{-2\sgn(\xi-\frac{j}{2})[(2\beta-\pi)(\xi-\frac{j}{2}+ih)]}}{1\!+e^{-2\sgn(\xi-\frac{j}{2})[(2\beta-\pi)(\xi-\frac{j}{2}+ih)]}}\rquad\ d\xi.\label{E3j}
\end{align}
Our problem has become the computation of the sum
\begin{equation}
\sum_{j\in\dsZ}J_j(\tau)\lambda^j=4e^{-\tau h}\Bigg[\sum_{j\in\dsZ}M_j(\tau)\lambda^j+\sum_{k=1}^{3}\sum_{j\in\dsZ}E^{(k)}_j(\tau)\lambda^j\Bigg].
\end{equation}
We will use the following scheme to compute the integrals (\ref{Mj})-(\ref{E3j}) . If $j>0$, we choose a positive $\delta$ such that $0<\delta<j/2$ and we consider
\begin{align}\label{j>0}
\int_{\dsR}f&=\Big[\int_{-\infty}^{-\delta}+\int_{-\delta}^{\delta}+\int_{\delta}^{\frac{j}{2}-\delta}+\int_{\frac{j}{2}-\delta}^{\frac{j}{2}+\delta}+\int_{\frac{j}{2}+\delta}^{+\infty}\Big]f=I+\clE_1+\II+\clE_2+\III.
\end{align}
Analogously, for negative $j$'s, we choose a positive $\delta$ such that $j/2<-\delta<0$ and we consider
\begin{align}\label{j<0}
\int_{\dsR}f&= \Big[\int_{-\infty}^{\frac{j}{2}-\delta}+\int_{\frac{j}{2}-\delta}^{\frac{j}{2}+\delta}+\int_{\frac{j}{2}+\delta}^{-\delta}+\int_{-\delta}^{\delta}f+\int_{\delta}^{+\infty}\Big]f=I^*+\clE^*_1+\II^*+\clE^*_2+\III^*.
\end{align}
We remark that the case $j=0$ is somehow special, but it could be treated in a similar way. Also, notice that the decomposition of the integrals above make sense even for $\delta=0$; this choice of $\delta$ will be the case in the computation of the sum of the $M_j$'s as we immediately see.
\bprop
There exist entire functions $\psi_i(\tau,\lambda),i=1,2,3,4$, such that
\begin{align}\label{Mj-final-1}\nonumber
4e^{-\tau h}&\sum_{j\in\dsZ}M_j(\tau)\lambda^j=4e^{-\tau
  h}\Bigg[\frac{e^{2\beta ih}}{i\tau+2\beta}+\frac{-e^{-2\beta
    ih}}{i\tau-2\beta}+\frac{-e^{2\beta ih}}{(i\tau+2\beta)(1-\lambda  e^{\frac{i\tau+\pi}{2}})}\\ \nonumber
    &+\frac{e^{-2\beta
    ih}}{(i\tau-2\beta)(1-\lambda e^{\beta-\pih})}+\frac{\psi_1(\lambda)}{(i\tau+2\beta)(1-\lambda
  e^{-(\beta-\pih)})}+\frac{\psi_2(\tau,\lambda)}{(i\tau-2\beta)(1-\lambda
  e^{\frac{i\tau-\pi}{2}})} \\ 
&\ \  +\frac{\psi_3(\tau,\lambda)}{(1-\lambda e^{\frac{i\tau-\pi}{2}})(1-\lambda e^{-(\beta-\pih)})}+\frac{\psi_4(\tau,\lambda)}{(1-\lambda e^{\beta-\pih})(1-\lambda e^{\frac{i\tau+\pi}{2}})}\Bigg],
\end{align}
where
\begin{align*}
&\psi_1(\lambda)= \lambda e^{2\beta ih}e^{-(\beta-\pih)};\qquad\ \ \ \psi_3(\tau,\lambda)=\lambda e^{-(\beta-\pih)}e^{2(\beta-\pi)ih}\lquad\frac{e^{\frac{i\tau}{2}+\beta-\pi}-1}{i\tau+2\beta-2\pi}\rquad;\\
&\psi_2(\tau,\lambda)=-\lambda e^{-2\beta ih} e^{\frac{i\tau-\pi}{2}};\qquad \psi_4(\tau,\lambda)=\lambda e^{-2(\beta-\pi)ih}e^{\beta-\pih}\lquad \frac{e^{\frac{i\tau}{2}-\beta+\pi}-1}{i\tau-2\beta+2\pi}\rquad.
\end{align*}
\eprop
\begin{proof} First of all, we have to compute each single $M_j(\tau)$. 
In this case we choose $\delta=0$ in (\ref{j>0}) and (\ref{j<0}) so that we do not have the error terms $\clE_1,\clE_2,\clE^*_1$ and $\clE^*_2$.
We begin focusing on the positive $j$'s. Therefore,
\begin{align}
&I=e^{2\beta ih}e^{-(2\beta-\pi)\frac{j}{2}}\int_{-\infty}^{-\delta}e^{(i\tau+2\beta)\xi}\ d\xi;\label{MjI}\\
&\II=e^{2(\beta-\pi)ih}e^{-(2\beta-\pi)\frac{j}{2}}\int_{\delta}^{\frac{j}{2}-\delta}e^{(i\tau+2\beta-2\pi)\xi}\ d\xi;\label{MjII}\\
&\III=e^{-2\beta ih}e^{(2\beta-\pi)\frac{j}{2}}\int_{\frac{j}{2}+\delta}^{+\infty}e^{(i\tau-2\beta)\xi}\ d\xi.\label{MjIII}
\end{align}
Choosing $\delta=0$ we obtain
\begin{align*}
&I=\frac{e^{2\beta ih}}{i\tau+2\beta}e^{-(2\beta-\pi)\frac{j}{2}};\qquad\II=\frac{e^{2(\beta-\pi)ih}}{i\tau+2\beta-2\pi}\lt e^{(i\tau-\pi)\frac{j}{2}}-e^{-(2\beta-\pi)\frac{j}{2}}\rt;\\
&\III=-\frac{e^{-2\beta i h}}{i\tau-2\beta}e^{(i\tau-\pi)\frac{j}{2}}.
\end{align*}
Summing up over the positive $j$'s we obtain
\begin{align}\label{Mj>0}
\sum_{j>0}M_j(\tau)\lambda^j&=\sum_{j>0}(I+\II+\III)\lambda^j\\ \nonumber
&=\frac{e^{2\beta ih}}{i\tau+2\beta}\!\lquad\!\frac{\lambda }{e^{\beta-\pih}-\lambda}\!\rquad\!-\frac{e^{-2\beta ih }}{i\tau-2\beta}\!\lquad\frac{\lambda }{e^{-\frac{i\tau-\pi}{2}}-\lambda}\rquad \\ \nonumber
&\hspace{20mm}+\!\frac{\lambda e^{-(\beta-\pih)}e^{2(\beta-\pi)ih}}{\big(1-\lambda e^{\frac{i\tau-\pi}{2}}\big)\lt1-\lambda e^{-(\beta-\pih)}\rt}\!\lquad\frac{e^{\frac{i\tau}{2}+\beta-\pi}-1}{i\tau+2\beta-2\pi}\rquad.
\end{align}

Notice that we do not have a singularity when $\tau\rightarrow 2\beta-2\pi$.

Analogously, using (\ref{j<0}), we obtain a result for negative $j$'s. Choosing again $\delta=0$, we obtain
\begin{align}\label{Mj<0}
\sum_{j<0}M_j(\tau)\lambda^j&=\sum_{j<0}(I^*+\II^*+\III^*)\lambda^j\\ \nonumber
&=\frac{e^{2\beta ih}}{i\tau+2\beta}\lquad\frac{1}{\lambda e^{\frac{i\tau+\pi}{2}}-1}\rquad-\frac{e^{-2\beta ih}}{i\tau-2\beta}\frac{1}{\lambda e^{\beta-\pih}-1}\\ \nonumber
&\hspace{20mm}+\frac{\lambda e^{-2(\beta-\pi)ih}e^{\beta-\pih}}{\lt\lambda e^{\beta-\pih}-1\rt\lt\lambda e^{\frac{i\tau+\pi}{2}}-1\rt}\!\lquad \frac{e^{\frac{i\tau}{2}-\beta+\pi}-1}{i\tau-2\beta+2\pi}\rquad
 .
\end{align}

It remains to compute $M_0(\tau)$but it is immediate to verify that
\begin{equation}\label{Mj=0}
M_0(\tau)=\frac{e^{2\beta ih}}{i\tau+2\beta}-\frac{e^{-2\beta ih}}{i\tau-2\beta}.
\end{equation}
Simplifying the notation a little bit we obtain \eqref{Mj-final-1} as we wished.
\end{proof}

At this point, we wish to evaluate the sums $\sum_{j\in\dsZ} E^{(k)}_{j}(\tau)\lambda^j$ for $k=1,2,3$. 
Recall that we are still supposing that $\re\tau\geq0$. We first introduce some domains which all are neighborhood of $\clD$, namely
\begin{align}
&\clD'=\big\{ (\tau,\lambda)\in \dsC^2:\left|\im\tau-\log|\lambda|^2\right|<2\pi,\left|\log|\lambda|^2\right|<2\beta-\frac{\pi}{2}\big\};\\
&\clD_{\infty,2\pi}=\lgra(\tau,\lambda)\in\dsC^2:\left|\im\tau-\log|\lambda|^2\right|<2\pi,|\lambda|>0\rgra;\\
&S_{2\beta+\frac{3}{2}\pi}=\big\{\tau\in\dsC:|\im\tau|<2\beta+\frac{3}{2}\pi\big\}.
\end{align}
\bprop\label{E1}
Let us consider
\[
E^{(1)}_j(\tau)=\int_{\dsR}\sigma(\xi)\frac{e^{i\tau\xi}}{e^{\pi|\xi|+(2\beta-\pi)|\xi-\frac{j}{2}|}}\lquad\frac{e^{-2\sgn(\xi)[\pi(\xi+ih)]}}{1+e^{-2\sgn(\xi)[\pi(\xi+ih)]}}\rquad\ d\xi,
\]
where $\sigma(\xi)=e^{-i\sgn(\xi)\pi h-i\sgn(\xi-j/2)(2\beta-\pi)h}$. Then
\begin{equation}\label{E1-tot-bis}
e^{-\tau h}\sum_{j\in\dsZ}E^{(1)}_{j}(\tau)\lambda^j=e^{-\tau h}\Bigg[\frac{\Psi^{(1)}(\tau,\lambda)}{e^{\beta-\pih}-\lambda}+\frac{\Psi^{(2)}(\tau,\lambda)}{e^{-(\beta-\pih)}-\lambda}+\Psi^{(3)}(\tau,\lambda)\Bigg],
\end{equation}
where $\Psi^{(j)}$ are holomorphic functions in a  neighborhood of $\overline{\clD}$, bounded together with all their derivatives as $|\re\tau|\rightarrow \infty$.
\eprop
\begin{proof}
Notice that with our choice of $h$ (see \eqref{h}), it results that $1+e^{-2\sgn(\xi)[\pi(\xi+ih)]}\neq0$ for every $\xi$. We decompose the integral defining $E^{(1)}_j$ as in (\ref{j>0}) and (\ref{j<0}), according to whether $j$ is positive or negative. 

We start analyzing the error terms $\clE_1$ and $\clE_2$. We have
\begin{align*}
&\clE_1= e^{2\beta ih}e^{-(2\beta-\pi)\frac{j}{2}}\int_{-\delta}^{0}e^{(i\tau+2\beta)\xi}\frac{e^{2\pi(\xi+ih)}}{1+e^{2\pi(\xi+ih)}}\ d\xi\\
&\ \ \ \ \ \ \ \ +e^{2(\beta-\pi)ih}e^{-(2\beta-\pi)\frac{j}{2}}\int_{0}^{\delta}e^{(i\tau+2\beta-2\pi)\xi}\frac{e^{-2\pi(\xi+ih)}}{1+e^{-2\pi(\xi+ih)}}\ d\xi,
\end{align*}
from which we deduce
\begin{align}\label{E1-e1}
&\sum_{j>0}\clE_1\lambda^j=\lquad\frac{\lambda e^{-(\beta-\pih)}}{1-\lambda e^{-(\beta-\pih)}}\rquad\\ \nonumber
&\ \times\Bigg[e^{2\beta ih}\int_{-\delta}^{0}e^{(i\tau+2\beta)\xi}\frac{e^{2\pi(\xi+ih)}}{1+e^{2\pi(\xi+ih)}}\ d\xi+e^{2(\beta-\pi)ih}\int_{0}^{\delta}e^{(i\tau+2\beta-2\pi)\xi}\frac{e^{-2\pi(\xi+ih)}}{1+e^{-2\pi(\xi+ih)}}\ d\xi\Bigg].
\end{align}
We conclude that 
\begin{equation}\label{E1-e1-bis}
e^{-\tau h}\sum_{j>0}\clE_1\lambda^j=\lquad\frac{ e^{-(\beta-\pih)}}{1-\lambda e^{-(\beta-\pih)}}\rquad e^{-\tau h}\Psi_{\clE_1}(\tau,\lambda),
\end{equation}
where $\Psi_{\clE_1}(\tau,\lambda)$ is entire and bounded together with all its derivatives as $|\re\tau|\!\rightarrow \infty$ and $\im\tau$ remains bounded.

To deal with $\clE_2$ is a little more complicated since the extremes of integration depend on $j$, but we cannot compute explicitly the integral in order to proceed with the sum in $j$. In fact, we have
\begin{align*}
\clE_2&=e^{2(\beta-\pi)ih}e^{-(2\beta-\pi)\frac{j}{2}}\int_{\frac{j}{2}-\delta}^{\frac{j}{2}}e^{(i\tau+2\beta-2\pi)\xi}\frac{e^{-2\pi(\xi+ih)}}{1+e^{-2\pi(\xi+ih)}}\ d\xi\\
&\ \ \  +e^{-2\beta ih}e^{(2\beta-\pi)\frac{j}{2}}\int_{\frac{j}{2}}^{\frac{j}{2}+\delta}e^{(i\tau-2\beta)\xi}\frac{e^{-2\pi(\xi+ih)}}{1+e^{-2\pi(\xi+ih)}}\ d\xi\\
&=A_j+B_j.
\end{align*}
We notice that 
$$
\frac{e^{-2\pi(\xi+ih)}}{1+e^{-2\pi(\xi+ih)}}= -\sum_{k>0}\lquad-e^{-2\pi(\xi+ih)}\rquad^k,\ \,\xi>0,
$$
where the series converges uniformly on compact sets with bounds uniform in $j>0$. This allows to interchange the order of integration and summation over $k$. Then
\begin{align*}
A_j&=-e^{2(\beta-\pi)ih}e^{-(2\beta-\pi)\frac{j}{2}}\int_{\frac{j}{2}-\delta}^{\frac{j}{2}}e^{(i\tau+2\beta-2\pi)\xi}\sum_{k>0}\lquad-e^{-2\pi(\xi+ih)}\rquad^k\ d\xi\\
&=-e^{2(\beta-\pi)ih}e^{-(2\beta-\pi)\frac{j}{2}}\sum_{k>0}\lquad -e^{-2\pi ih} \rquad^k\int_{\frac{j}{2}-\delta}^{\frac{j}{2}} e^{(i\tau+2\beta-2\pi-2\pi k)\xi}\ d\xi.
\end{align*}
Summing up on positive $j$'s, we obtain 
\begin{align}\label{E1-e2-1}\nonumber
\sum_{j>0}&A_j\lambda^j=-e^{2(\beta-\pi)ih}\sum_{j>0}\lambda^j e^{-(2\beta-\pi)\frac{j}{2}}\sum_{k>0}\!\lquad -e^{-2\pi ih} \rquad^k\!\!\int_{\frac{j}{2}-\delta}^{\frac{j}{2}} e^{(i\tau+2\beta-2\pi-2\pi k)\xi} d\xi\\ \nonumber
&=-e^{2(\beta-\pi)ih}\!\sum_{k>0}\lquad -e^{-2\pi ih}\rquad^k\sum_{j>0}\lambda^j e^{(i\tau-\pi)\frac{j}{2}}\int_{-\delta}^{0} e^{(i\tau+2\beta-2\pi)\xi}e^{-2\pi k\lt\xi+\frac{j}{2}\rt}d\xi\\ \nonumber
&=-e^{2(\beta-\pi) ih}\sum_{k>0}\lquad-e^{-2\pi ih}\rquad^{k}\lquad\frac{\lambda e^{\frac{i\tau-\pi-2\pi k}{2}}}{1-\lambda e^{\frac{i\tau-\pi-2\pi k}{2}}} \rquad\int_{-\delta}^{0}e^{(i\tau+2\beta-2\pi-2\pi k)\xi}d \xi \\ 
&=-e^{2(\beta-\pi) ih}\sum_{k>0}\lquad-e^{-2\pi ih}\rquad^{k}\lquad\frac{\lambda}{e^{\frac{\pi+2\pi k-i\tau}{2}}-\lambda } \rquad h^{(k)}_1(\tau),
\end{align}

where $h^{(k)}_1(\tau)$  is an entire function such that
$$
\left| h^{(k)}_1(\tau)\right|\leq c_{\delta} e^{2\pi k\delta} \lquad\frac{1-e^{-\delta(2\beta-\im\tau)}}{2\beta-\im\tau}\rquad.
$$
Notice that we do not have a singularity when $\im\tau\rightarrow 2\beta$. The convergence of the sum in $j$ is guaranteed when $\left|\lambda e^{\frac{i\tau-\pi-2\pi k}{2}}\right|<1$ and this last condition is satisfied for every positive $k$ when the pair $(\tau,\lambda)$ belongs to $\clD_{\infty,2\pi}$ .

We still have to study $\sum_{j>0}B_j\lambda^j$. We have
$$
B_j=-e^{-2\beta ih}e^{(2\beta-\pi)\frac{j}{2}}\sum_{k>0}\lquad-e^{-2\pi ih}\rquad^k\int_{\frac{j}{2}}^{\frac{j}{2}+\delta}e^{(i\tau-2\beta-2\pi k)\xi}\ d\xi.
$$
Then,
\begin{align}\label{E1-e2-2}\nonumber
\sum_{j>0}&B_j\lambda^j=-e^{-2\beta ih}\sum_{j>0}\lambda^j e^{(2\beta-\pi)\frac{j}{2}}\sum_{k>0}\lquad-e^{-2\pi ih}\rquad^k\int_{\frac{j}{2}}^{\frac{j}{2}+\delta}e^{(i\tau-2\beta-2\pi k)\xi}\ d\xi\\ \nonumber
&=-e^{-2\beta ih}\sum_{k>0}\lquad-e^{-2\pi ih}\rquad^k\sum_{j>0}\lambda^j e^{(i\tau-\pi-2\pi k)\frac{j}{2}}\int_{0}^{\delta}e^{(i\tau-2\beta-2\pi k)\xi}\ d\xi\\
&=-e^{-2\beta ih}\sum_{k>0}\lquad-e^{-2\pi ih}\rquad^k\lquad \frac{\lambda}{e^{\frac{\pi+2\pi k-i\tau}{2}}-\lambda }\rquad h^{(k)}_{2}(\tau).
\end{align}
Here $h^{(k)}_2(\tau)$ is an entire function such $\left|h^{(k)}_2\right|<\lquad\frac{1-e^{-\delta(\im\tau+2\beta)}}{\im\tau+2\beta}\rquad$ and we use the fact that $\left|\lambda e^{\frac{i\tau-\pi-2\pi k}{2}}\right|<1$ for every positive $k$. In conclusion,
\begin{align}\label{E1-e2} 
\sum_{j>0}\clE_2\lambda^j=-\sum_{k>0}[-e^{-2\pi ih}]^k & \lquad\frac{\lambda}{e^{\frac{\pi+2\pi k-i\tau}{2}}-\lambda }\rquad\lquad e^{2(\beta-\pi)ih}h_1^{(k)}(\tau)+e^{-2\beta ih}h_2^{(k)}(\tau)\rquad.
\end{align}
We want to prove that this sum on $k$ converges to a function holomorphic on the domain $\clD_{\infty,2\pi}$. To prove this is enough to assume $\delta<1/2$ and to notice that, for fixed $M>0$, it is possible to select $k_0$ large enough so that for all $k\geq k_0$, when $(\tau,\lambda)\in \clD_{\infty,2\pi}$ with $\im\tau\leq M$ and $|\lambda|\leq e^M$, we have that
$$
\left|e^{\frac{\pi+2\pi k-i\tau}{2}}-\lambda\right|\geq c e^{\pi k}.
$$ 
Thus, the sum in $k$ is uniform on the fixed compact set. Therefore, we have
\begin{equation}\label{E1-e2-bis}
e^{-\tau h}\sum_{j>0}\clE_2\lambda^j=e^{-\tau h}\lquad e^{2(\beta-\pi)ih}\Psi^{(1)}_{\clE_2}(\tau,\lambda)+e^{-2\beta ih}\Psi^{(2)}_{\clE_2}(\tau,\lambda)\rquad,
\end{equation}
where $\Psi^{(i)}_{\clE_2}(\tau,\lambda)$ are holomorphic on $\clD_{\infty,2\pi}$, bounded together with their derivatives as $|\re\tau|\rightarrow \infty$ and $\im\tau$ and $\lambda$ remain bounded.
We took care of the error terms $\clE_1$ and $\clE_2$; using the same strategy, we now study $I,\II$ and $\III$. We have
\begin{align*}
&I=e^{2\beta
ih}e^{-(2\beta-\pi)\frac{j}{2}}\int_{-\infty}^{-\delta}e^{(i\tau+2\beta)\xi}
\frac{e^{2\pi(\xi+ih)}}{1+e^{2\pi(\xi+ih)}}\ d\xi;\\
&\II=e^{2(\beta-\pi)ih}e^{-(2\beta-\pi)\frac{j}{2}}\int_{\delta}^{\frac{j}{2}
-\delta}
e^{(i\tau+2\beta-2\pi)\xi}\frac{e^{-2\pi(\xi+ih)}}{1+e^{-2\pi(\xi+ih)}}\ d\xi\\
&\ \ \ 
=-e^{2(\beta-\pi)ih}e^{-(2\beta-\pi)\frac{j}{2}}\sum_{k>0}\lquad
-e^{-2\pi
ih}\rquad^{k}\int_{\delta}^{\frac{j}{2}-\delta}e^{(i\tau+2\beta-2\pi-2\pi
k)\xi}\ d\xi;\\
&\III=e^{-2\beta ih}
e^{(2\beta-\pi)\frac{j}{2}}\int_{\frac{j}{2}+\delta}^{+\infty}
e^{(i\tau-2\beta)\xi}\frac{e^{-2\pi(\xi+ih)}}{1+e^{-2\pi(\xi+ih)}}\ d\xi\\
&\ \ \ =-e^{-2\beta ih}e^{(2\beta-\pi)\frac{j}{2}}\sum_{k>0}\lquad
-e^{-2\pi
ih}\rquad^{k}\int_{\frac{j}{2}+\delta}^{+\infty}e^{(i\tau-2\beta-2\pi k)\xi}\
d\xi.
\end{align*}
Then, if $\left|\lambda e^{-(\beta-\pih)}\right|<1$, we obtain
\begin{align}\label{E1-I}\nonumber
\sum_{j>0} I\lambda_j&= e^{2\beta ih}\lquad \frac{\lambda
e^{-\frac{2\beta-\pi}{2}}}{1-\lambda
e^{-\frac{2\beta-\pi}{2}}}\rquad\lquad\int_{-\infty}^{-\delta}e^{
(i\tau+2\beta)\xi}
\frac{e^{2\pi(\xi+ih)}}{1+e^{2\pi(\xi+ih)}}\ d\xi\rquad\\ \nonumber
&=e^{2\beta ih}\lquad \frac{\lambda}{e^{(\beta-\pih)}-\lambda
}\rquad\sum_{k>0}\lquad -e^{2\pi ih}\rquad^{k}\lquad \frac{e^{-\delta(i\tau+2\beta+2\pi k)}}{i\tau+2\beta+2\pi k}\rquad\\
&= e^{2\beta ih}\lquad\frac{\Psi_{I}(\tau,\lambda)}{e^{(\beta-\pih)}-\lambda}\rquad
\end{align}
where $\Psi_{I}$ is holomorphic in a neighborhood of $\overline{\clD}$. In fact, let us consider 
\begin{equation}
\label{D'}\clD'=\lgra (\tau,\lambda)\in\dsC^2:\left| \im\tau-\log|\lambda|^2\right|<2\pi,\left|\log|\lambda|^2\right|<2\beta-\frac{\pi}{2}\rgra.
\end{equation}
Then, on this set, the holomorphicity of $\Psi_{I}(\tau,\cdot)$ as a function of $\lambda$ for every $\tau$ fixed is obvious, whereas for the holomorphicity of $\Psi_{I}(\cdot,\lambda)$ we have 
\begin{align*}
\left|\lquad -e^{2\pi ih}\rquad^{k}\frac{e^{-\delta(i\tau+2\beta+2\pi k)}}{i\tau+2\beta+2\pi k} \right|&\leq e^{\delta(\im\tau-2\beta)}\frac{e^{-2\pi k\delta}}{\lquad (\re\tau)^2+(2\beta+2\pi k-\im\tau)^2\rquad^{\frac{1}{2}}}\\
&\leq C e^{\delta(\im\tau-2\beta)}e^{-2\pi k \delta}.
\end{align*}
This is true because $\im\tau<2\beta+\frac{3}{2}\pi<2\beta+2\pi k$ for all $k\geq1$. Thus, we have uniform convergence and we can conclude that
\begin{equation}
e^{-\tau h}\sum_{j>0}I\lambda ^j=e^{-\tau h} e^{2\beta ih}\lquad\frac{\Psi_{I}(\tau,\lambda)}{e^{(\beta-\pih)}-\lambda}\rquad,  
\end{equation}
where $\Psi_I(\tau,\lambda)$ is holomorphic in $\clD'$.

About $\II$, notice that
\begin{align*}
\II\!=-e^{2(\beta-\pi)ih-(2\beta-\pi)\frac{j}{2}}\sum_{k>0}\!\lquad-e^{-2\pi ih}\rquad^k\! &e^{(i\tau+2\beta-2\pi-2\pi k)\delta}\!\lgra\frac{e^{(i\tau+2\beta-2\pi-2\pi k)(\frac{j}{2}-2\delta)}-1}{i\tau+2\beta-2\pi-2\pi k}\!\rgra
\end{align*}
and we do not have a singularity when $i\tau+2\beta-2\pi-2\pi k$ tends to $0$.

If we suppose again $\left|\lambda
e^{\frac{i\tau-\pi-2\pi k}{2}}\right|<1$, we get
\begin{align}\label{E1-II} \nonumber
\sum_{j>0}\II\lambda^j=\frac{-e^{2(\beta-\pi)ih}}{e^{\beta-\pih}-\lambda}&\sum_{k>0}\frac{\lambda\lquad-e^{-2\pi
ih}\rquad^k  e^{(i\tau+2\beta-2\pi-2\pi k)\delta}}{i\tau+2\beta-2\pi-2\pi k}\\ \nonumber
&\qquad\times \Bigg[\frac{e^{-2(i\tau+2\beta-2\pi-2\pi k)\delta}(e^{\beta-\pih}-\lambda)-(e^{\frac{2\pi k+\pi+i\tau}{2}}-\lambda)}{e^{\frac{2\pi k+\pi-i\tau}{2}}-\lambda}\Bigg].
\end{align}

We want to say something more about the sum in $k$. For each $M>0$ we can select $k_0$ such that for every $k>k_0$ and $(\tau,\lambda )$ with $|\im\tau|<M$ and $|\lambda|<e^M$, we have
$$
\left|e^{\frac{2\pi k+2\pi-i\tau}{2}}-\lambda\right|\geq e^{\frac{2\pi(k+1)-M}{2}}-e^M\geq\frac{1}{2}e^{\pi k},
$$
so that the series in $k$ converges uniformly on the fixed compact set and we conclude that
\begin{equation}\label{E1-II-bis}
e^{-\tau h}\sum_{j>0}\II\lambda^j= -e^{2(\beta-\pi)ih}e^{-\tau h}\lquad\frac{\Psi_{\II}(\tau,\lambda)}{e^{\beta-\pih}-\lambda}\rquad
\end{equation}
where $\Psi_{\II}$ is a function holomorphic in $\clD_{\infty,2\pi}$.

Finally, for $(\tau,\lambda)$ in $\clD'$, it holds for every positive $k$ that $\left|\lambda e^{\frac{i\tau-\pi-2\pi k}{2}}\right|<1$, so the sum of the $III$'s results to be
\begin{align}\label{E1-III}
\sum_{j>0}\III\lambda^{j}&= e^{-2\beta ih}\sum_{k>0}\lquad-e^{-2\pi
ih}\rquad^k\lquad\frac{e^{\delta(i\tau-2\beta-2\pi k)}}{i\tau-2\beta-2\pi
k}\rquad\lquad \frac{\lambda e^{\frac{i\tau-\pi-2\pi k}{2}}}{1-\lambda
e^{\frac{i\tau-\pi-2\pi k}{2}}}\rquad.
\end{align}
We have to discuss the sum over $k$. We notice that
\begin{align*}
\left| \frac{\lambda e^{\frac{i\tau-\pi-2\pi k}{2}}}{1-\lambda
e^{\frac{i\tau-\pi-2\pi k}{2}}}\right|&\leq \frac{|\lambda| e^{\frac{-\im\tau-\pi-2\pi k}{2}}}{e^{\frac{-\im\tau-\pi-2\pi k}{2}}\left|\im\lquad \lambda e^{\frac{i \re\tau}{2}}\rquad\right|}=\frac{|\lambda| e^{\frac{-\im\tau-\pi}{2}}}{e^{\frac{-\im\tau-\pi}{2}}\left|\im\lquad \lambda e^{\frac{i \re\tau}{2}}\rquad\right|}.
\end{align*}
So $\left|\frac{\lambda e^{-\frac{i\tau-\pi-2\pi k}{2}}}{1-\lambda e^{-\frac{i\tau-\pi-2\pi k}{2}}}\right|$ is uniformly bounded in k. Moreover, since $\im\tau>-2\beta-\frac{3}{2}\pi>-2\beta-2\pi k$ for every positive $k$, the series $\sum\limits_{k>0}\lquad -e^{-2\pi ih} \rquad^{k}\lquad \frac{e^{\delta(i\tau-2\beta-2\pi k)}}{i\tau-2\beta-2\pi k}\rquad$ converges uniformly in $\tau$ and we conclude that
\begin{equation}\label{E1-III-bis}
e^{-\tau h}\sum_{j>0}\III\lambda^j= e^{-2\beta ih}e^{-\tau h}\Psi_{\III}(\tau,\lambda),
\end{equation}
where $\Psi_{\III}(\tau,\lambda)$ is holomorphic in $D'$. We remark that the functions $\Psi_I,\Psi_{II}$ and $\Psi_{\III}$ are bounded together with all their derivative as $|\re\tau|\rightarrow \infty$ and $\im\tau$ and $\lambda$ remain bounded.

We now focus on the sum over negative $j$'s. Again, we start analyzing the error terms $\clE_1^*$ and $\clE_2^*$. We have
\begin{align*}
\clE_1^*&= e^{2\beta ih}e^{-(2\beta-\pi)\frac{j}{2}}\int_{\frac{j}{2}-\delta}^{\frac{j}{2}} \frac{e^{(i\tau+2\beta)\xi}e^{2\pi(\xi+ih)}}{1+e^{2\pi(\xi+ih)}}\ d\xi\\
&\hspace{38mm}+e^{-2(\beta-\pi)ih}e^{(2\beta-\pi)\frac{j}{2}}\int_{\frac{j}{2}}^{\frac{j}{2}+\delta} \frac{e^{(i\tau-2\beta+2\pi)\xi}e^{2\pi(\xi+ih)}}{1+e^{2\pi (\xi+ih)}}d \xi\\
&=C_j+D_j.
\end{align*}
If we suppose $(\tau,\lambda)\in\clD_{\infty,2\pi}$, then $\left|\lambda e^{\frac{i\tau+\pi+2\pi k}{2}}\right|>1$ for every positive $k$, so  we obtain
$$
\sum_{j<0}C_j\lambda^j=-e^{2\beta ih}\sum_{k>0}\frac{\lquad -e^{2\pi ih}\rquad^k}{\lambda e^{\frac{i\tau+\pi+2\pi k}{2}}-1}\int_{-\delta}^{0}e^{(i\tau+2\beta+2\pi k)\xi}\ d\xi.
$$
Now, since $(\tau,\lambda)$ is in $D_{\infty,2\pi}$, it holds $|\lambda|>e^{\frac{\im\tau}{2}-\pi}>e^{\frac{\im\tau}{2}-\frac{3}{2}\pi}\geq e^{\frac{\im\tau}{2}-\pih-\pi k}$ for every positive $k$, so
\begin{align*}
\left|\lambda e^{\frac{i\tau+\pi+2\pi k}{2}-1}\right|&=e^{\frac{\pi+2\pi k-\im\tau}{2}}\left|\lambda-e^{-\frac{i\tau+\pi+2\pi k}{2}}\right|
\geq e^{\frac{\pi+2\pi k-\im\tau}{2}}\lt e^{\frac{\im\tau}{2}-\pi}-e^{\frac{\im\tau-3\pi}{2}}\rt
>0.
\end{align*}
Using this estimates and the fact that $\left|\int_{-\delta}^{0}e^{(i\tau+2\beta+2\pi k)\xi}d \xi\right|$ is uniformly bounded in $k$, we can conclude that
$$
\sum_{j<0}C_j\lambda^j=e^{2\beta ih}\Psi_{\clE_1^*}^{(1)}(\tau,\lambda),
$$
where $\Psi_{\clE_1^*}^{(1)}$ is holomorphic in $\clD_{\infty,2\pi}$.
Similarly,
$$
\sum_{j<0}D_j\lambda^j=-e^{-2(\beta-\pi)ih}\sum_{k>0}\frac{\lquad -e^{2\pi ih}\rquad^k}{\lambda e^{\frac{i\tau+\pi+2\pi k}{2}}-1}\int_{0}^\delta e^{(i\tau-2\beta+2\pi+2\pi k)\xi}\ d\xi,
$$
Arguing as before, if in addition we suppose $\delta<\frac{1}{2}$, we obtain
$$
\sum_{j<0}D_j\lambda^j=e^{-2(\beta-\pi)ih}\Psi_{\clE^*_1}^{(2)}(\tau,\lambda),
$$
where $\Psi_{\clE^*_1}^{(2)}$ is holomorphic in $\clD_{\infty,2\pi}$.

In conclusion we obtain
\begin{equation}\label{E1-e1*-bis}
e^{-\tau h}\sum_{j<0}\clE_{1}^*\lambda^j= e^{-\tau h}\lquad e^{2\beta ih}\Psi^{(1)}_{\clE_1^*}(\tau,\lambda)+e^{-2(\beta-\pi)ih}\Psi^{(2)}_{\clE_1^*}(\tau,\lambda)\rquad,
\end{equation}
where $\Psi^{(i)}_{\clE_1^*}(\tau,\lambda)$ are holomorphic on $D_{\infty,2\pi}$.

For $\clE^*_2$ it results
\begin{align*}
\clE^*_2&= e^{-2(\beta-\pi)ih}e^{(2\beta-\pi)\frac{j}{2}}\int_{-\delta}^{0}\frac{e^{(i\tau-2\beta+2\pi)\xi}e^{2\pi(\xi+ih)}}{1+e^{2\pi(\xi+ih)}}\ d\xi\\
&\hspace{43mm}+e^{-2\beta ih}e^{(2\beta-\pi)\frac{j}{2}}\int_{0}^{\delta}\frac{ e^{(i\tau-2\beta)\xi}e^{-2\pi(\xi+ih)}}{1+e^{-2\pi(\xi+ih)}}\ d\xi.
\end{align*}
It follows, for $\left|\lambda e^{\beta-\pih}\right|>1$,
\begin{align}\label{E1-e2*}
&\sum_{j<0}\clE^*_2\lambda^j= \lquad\frac{1}{\lambda e^{\frac{2\beta-\pi}{2}}-1}\rquad\\
&\hspace{5mm}\times\Bigg[ e^{-2(\beta-\pi)ih}\int_{-\delta}^{0}\frac{e^{(i\tau-2\beta+2\pi)\xi}e^{2\pi(\xi+ih)}}{1+e^{2\pi(\xi+ih)}}\ d\xi+e^{-2\beta ih}\int_{0}^{\delta} \frac{e^{(i\tau-2\beta)\xi}e^{-2\pi(\xi+ih)}}{1+e^{-2\pi(\xi+ih)}}\ d\xi\Bigg].
\end{align}
We conclude that
\begin{equation}\label{E1-e2*-bis}
e^{-\tau h}\sum_{j<0}\clE^*_2\lambda^j=e^{-\tau h}\lquad\frac{\Psi_{\clE_2^*}(\tau)}{\lambda e^{\beta-\frac{\pi}{2}}-1}\rquad,
\end{equation}
where $\Psi_{\clE_2^*}$ is entire.

Let us see what happens with the principal terms $I^*,\II^*$ and $\III^*$. We have
\begin{align*}
&I^* =-e^{2\beta ih} e^{-(2\beta-\pi)\frac{j}{2}}\int_{-\infty}^{\frac{j}{2}-\delta}e^{(i\tau+2\beta)\xi}\sum_{k>0}\lquad-e^{2\pi(\xi+ih)}\rquad^{k}\ d\xi;\\
&\II^*= - e^{-2(\beta-\pi)ih} e^{(2\beta-\pi)\frac{j}{2}}\int_{\frac{j}{2}+\delta}^{-\delta}e^{(i\tau+2\pi-2\beta)\xi}\sum_{k>0}\!\lquad-e^{2\pi(\xi+ih)}\rquad^{k}d\xi;\\
&\III^*= e^{-2\beta ih} e^{(2\beta-\pi)\frac{j}{2}}\int_{\delta}^{+\infty} e^{(i\tau-2\beta)\xi}\frac{e^{-2\pi(\xi+ih)}}{1+e^{-2\pi(\xi+ih)}}\ d\xi. 
\end{align*}
Then, if we suppose $(\tau,\lambda)$ in $\clD'$, it holds $\left|\lambda e^{\frac{i\tau+\pi+2\pi k}{2}}\right|>1$ for every positive $k$, so
\begin{align}\label{E1-I*}
&\sum_{j<0}I^*\lambda^j= -e^{2\beta ih}\sum_{k>0}\lquad -e^{2\pi ih}\rquad^k\frac{e^{-\delta(i\tau+2\beta+2\pi k)}}{i\tau+2\beta+2\pi k}\lquad\frac{1}{\lambda e^{\frac{i\tau+\pi+2\pi k}{2}}-1}\rquad;\\  \label{E1-II*}
&\sum_{j<0}\II^*\lambda ^j = \sum_{k>0}\frac{-e^{2(\beta-\pi)ih}\lquad-e^{2\pi ih}\rquad^k}{i\tau-2\beta+2\pi+2\pi k}\lquad \frac{e^{-\delta(i\tau-2\beta+2\pi+2\pi k)}}{\lambda e^{\frac{2\beta-\pi}{2}}-1}-\frac{e^{\delta(i\tau-2\beta+2\pi+2\pi k)}}{\lambda e^{\frac{i\tau+\pi+2\pi k}{2}}-1}\rquad;\\ \label{E1-III*}
&\sum_{j<0}\III^*\lambda^j= \frac{e^{-2\beta ih}}{\lambda e^{\frac{2\beta-\pi}{2}}-1}\sum_{k>0}\lquad -e^{-2\pi ih}\rquad^{k}\frac{e^{\delta(i\tau-2\beta-2\pi k)}}{i\tau-2\beta-2\pi k}.
\end{align}
Notice that $(\tau,\lambda)\in\clD'$ implies that $i\tau+2\beta+2\pi k\neq 0$ for every positive $k$, therefore
\begin{equation}\label{E1-I*-bis}
e^{-\tau h}\sum_{j<0}I^*\lambda^j=-e^{2\beta ih}e^{-\tau h}\Psi_{I^*}(\tau,\lambda),
\end{equation}
where $\Psi_{I^*}(\tau,\lambda)$ is holomorphic in $D'$. Analogously, for $\sum_{j<0}\III^*\lambda^{j}$ we have
\begin{align*}
\frac{e^{\delta(i\tau-2\beta-2\pi k)}}{i\tau-2\beta-2\pi l}&\leq \frac{e^{\delta(-\im\tau-2\beta)}e^{-2\pi\delta k}}{\lquad (\re\tau)^2-(\im\tau+2\beta+2\pi k)^2\rquad^{\frac{1}{2}}}\leq C e^{-2\pi\delta k},
\end{align*}
where the last inequality is true since $\im\tau>-2\beta-\frac{3}{2}\pi>-2\beta-2\pi k$ for every $k\geq 1$.
So
\begin{equation}\label{E1-III*-bis}
e^{-\tau h}\sum_{j<0}\III^*\lambda^j=e^{2\beta ih}\lquad \frac{e^{-\tau h}\Psi_{\III^*}(\tau)}{\lambda e^{\beta-\frac{\pi}{2}}-1}\rquad,
\end{equation}
where $\Psi_{\III^*}(\tau)$ is holomorphic in $S_{2\beta+\frac{3}{2}\pi}$. About (\ref{E1-II*}) we notice that we do not have a singularity when $i\tau-2\beta+2\pi+2\pi k\rightarrow 0$. Then, for every $M>0$ and $(\tau,\lambda)\in D_{\infty,2\pi}$ such that $e^M>|\lambda|>e^{-M}$ and $|\im\tau|<M$ we can choose $k_0$ such that for every $k>k_0$ it holds
$$
\left| \lambda e^{\frac{i\tau+\pi+2\pi k}{2}}-1\right|\geq e^{-M}e^{\frac{-M+\pi+2\pi k}{2}}-1\geq \frac{1}{2}e^{\pi k}.
$$
Using this last estimate we can conclude that
\begin{equation}\label{E1-II*-bis}
e^{-\tau h}\sum_{j<0}\II^*\lambda^j=-e^{2(\beta-\pi)ih}\lquad
\frac{e^{-\tau h}\Psi_{II^*}(\tau,\lambda)}{\lambda
  e^{\beta-\pih}-1}\rquad 
\end{equation}
where $\Psi_{II^*}(\tau,\lambda)$ is holomorphic on $D_{2\pi,\infty}$.
It remains to study the term $E_0^{(1)}(\tau)$. Using some of the same arguments we used before it is possible to conclude that $E_0^{(1)}(\tau)$ is a holomorphic function in $S_{2\beta+\frac{3}{2}\pi}$. We remark that all the functions $\Psi_*$ are bounded together with all their derivatives as $|\re\tau|\rightarrow \infty$ and $\im\tau$ and $\lambda$ remain bounded. Finally, we deduce \eqref{E1-tot-bis} from equations (\ref{E1-e1-bis}), (\ref{E1-e2-bis}), \eqref{E1-I}, \eqref{E1-II-bis}, \eqref{E1-III-bis}, \eqref{E1-e1*-bis}, \eqref{E1-e2*-bis},  \eqref{E1-I*-bis}, \eqref{E1-III*-bis} and \eqref{E1-II*-bis}. 
\end{proof}
It remains to compute the sums $\sum_{j\in\dsZ}E^{(k)}_j(\tau)\lambda^j, k=2,3$. In order to keep the length of this work contained, we do not include the proofs since are similar to the one of Proposition \ref{E1} and we only state the results.
We refer the reader to \cite{MonThesis} for full details. 

We recall that 
$$
E^{(2)}_{j}(\tau)=\int_{\dsR}\sigma(\xi) \frac{e^{i\tau \xi}}{e^{\pi|\xi|+(2\beta-\pi)|\xi-\frac{j}{2}|}}\frac{e^{-2\sgn\lt\xi-\frac{j}{2}\rt\lquad(2\beta-\pi)(\xi-\frac{j}{2}+ih)\rquad}}{1+e^{-2\sgn\lt\xi-\frac{j}{2}\rt\lquad(2\beta-\pi)(\xi-\frac{j}{2}+ih)\rquad}}d \xi
$$
and 
\begin{align*}
E^{(3)}_{j}(\tau)=\int_{\dsR}\sigma(\xi) \frac{e^{i\tau
    \xi}}{e^{\pi|\xi|+(2\beta-\pi)|\xi-\frac{j}{2}|}}&\frac{e^{-2\sgn(\xi)\pi(\xi+ih)}}{1+e^{-2\sgn(\xi)\pi(\xi+ih)}}\\
    &\hspace{7mm}\times\frac{e^{-2\sgn\lt\xi-\frac{j}{2}\rt\lquad(2\beta-\pi)(\xi-\frac{j}{2}+ih)\rquad}}{1+e^{-2\sgn\lt\xi-\frac{j}{2}\rt\lquad(2\beta-\pi)(\xi-\frac{j}{2}+ih)\rquad}}d \xi,
\end{align*}
where 
$$
\sigma(\xi)=e^{-i\sgn(\xi)\pi h}e^{-i\sgn(\xi-\frac{j}{2})(2\beta-\pi)h}.
$$
Then, it holds the following proposition.
\bprop \label{E2}
There exist  $\Phi^{(k)}$, $k=1,2,3$, and $\Theta$ holomorphic functions in a  neighborhood of $\overline{\clD}$, bounded together with all their derivatives as $|\re\tau|\rightarrow \infty$ such that 
\begin{equation}\label{E2-tot-bis}
e^{-\tau h}\sum_{j\in\dsZ}E^{(2)}_{j}(\tau)\lambda^j=e^{-\tau h}\lquad\frac{\Phi^{(1)}(\tau,\lambda)}{1-\lambda e^{\frac{i\tau-\pi}{2}}}+\frac{\Phi^{(2)}(\tau,\lambda)}{\lambda e^{\frac{i\tau+\pi}{2}}-1}+\Phi^{(3)}(\tau,\lambda)\rquad,
\end{equation}
and
\begin{equation}\label{E3-tot-bis}
e^{-\tau h}\sum_{j\in\dsZ}E^{(3)}_{j}\lambda^j=e^{-\tau h}\Theta(\tau,\lambda),
\end{equation}  
\eprop
We finally have all the ingredients to prove Theorem \ref{t:kernel}. 
\begin{proof}[Proof of Theorem \ref{t:kernel}]
Recalling that $\tau=w_1-\overline{z_1}$ and $\lambda=w_2\overline{z_2}$ and assuming that $\re (w_1-\overline{z_1})\geq0$ we deduce \eqref{expansion} from \eqref{sum-Ij}, \eqref{sum-Rj}, \eqref{Mj-final-1}, \eqref{E1-tot-bis}, \eqref{E2-tot-bis} and \eqref{E3-tot-bis}. We point out that, in order to obtain a shorter formula for \eqref{expansion}, we grouped together the first and the fifth main terms of \eqref{Mj-final-1} with the first main term of \eqref{E1-tot-bis}, whereas we grouped the second and the fourth term of \eqref{Mj-final-1} with the second term of \eqref{E1-tot-bis}. 
The conclusion for $\re (w_1-\overline{z_1})<0$ is similarly obtained.
\end{proof}
\section{Comparison with the Bergman kernel}\label{comparison}
In this section we prove Theorem \ref{t:comparison}. The conclusion follows by an explicit computation of the derivatives of the kernel $K_{D'_\beta}(w,z)$.
\begin{proof}
Let $h$ be fixed such that both \eqref{expansion} and \eqref{ExpansionBergman} hold. We prove the theorem explicitly only for $\frac{\p}{\p w_1} K_{D'_\beta}$; the proof for the others cases is similarly obtained. 

 It is not hard to verify that there exist functions $\tilde{\rho}_1,\tilde{\rho}_2$ and $\mathcal{E}$ with the same properties of $\rho_1, \rho_2$ and $E$ in Theorem \ref{t:kernel} and functions $\widetilde{G}_k, k=1,\ldots,8$ and $\widetilde{\mathcal{E}}$ with the same properties of the functions $F_k, k=1,\ldots,8$ in Theorem \ref{KP} such that 
\bequa\label{DerivativeKernel}
\frac{\p}{\p w_1}K_{D'_\beta}=e^{-\sgn[\re(w_1-\overline{z_1})]\frac{(w_1-\overline{z_1})\nu_\beta}{2}}\mathcal{K}(w,z)+e^{-\sgn[\re(w_1-\overline{z_1})](w_1-\overline{z_1}) h}\widetilde{\mathcal{K}}(w,z)
\eequa
where
\begin{align*}
  \mathcal{K}(w,z)=\frac{\widetilde{\rho}_1(w,z)}{\big[e^{\frac{\pi-i(w_1-\overline{z_1})}{2}}-w_2\overline{z_2}\big]^2}+\frac{\widetilde{\rho}_2(w,z)}{\big[e^{-\frac{\pi+i(w_1-\overline{z_1})}{2}}-w_2\overline{z_2}\big]^2}+\mathcal{E}(w,z)
\end{align*}
and
\begin{align*}
 &\widetilde{\mathcal K}(w,z)\\
 &=\frac{\widetilde G_1(w,z)}{\big[i(w_1-\overline{z_1})+2\beta\big]^2\big[e^{\beta-\pih}-w_2\overline{z_2}\big]}+\frac{\widetilde G_2(w,z)}{\big[e^{-\frac{i(w_1-\overline{z_1})+\pi}{2}}-w_2\overline{z_2}\big]^2\big[i(w_1-\overline{z_1})+2\beta\big]^2}\\
  &\ \  +\frac{\widetilde G_3(w,z)}{\big[e^{\frac{\pi-i(w_1-\overline{z_1})}{2}}-w_2\overline{z_2}\big]^2\big[e^{\beta-\pih}-w_2\overline{z_2}\big]}+\frac{\widetilde G_4(w,z)}{\big[e^{\frac{\pi-i(w_1-\overline{z_1})}{2}}-w_2\overline{z_2}\big]^2\big[i(w_1-\overline{z_1})-2\beta\big]^2} \\
  &\ \ +\frac{\widetilde G_5(w,z)}{\big[i(w_1-\overline{z_1})-2\beta\big]^2\big[e^{-(\beta-\pih)}-w_2\overline{z_2}\big]}+\frac{\widetilde G_6(w,z)}{\big[e^{-\frac{i(w_1-\overline{z_1})+\pi}{2}}-w_2\overline{z_2}\big]^2\big[e^{-(\beta-\pih)}-w_2\overline{z_2}\big]}\\
  &\ \ +\frac{\widetilde G_7(w,z)}{\big[e^{\frac{\pi-i(w_1-\overline{z_1})}{2}}-w_2\overline{z_2}\big]^2}+\frac{\widetilde G_8(w,z)}{\big[e^{-\frac{i(w_1-\overline{z_1})+\pi}{2}}-w_2\overline{z_2}\big]^2}+\widetilde{\clE}(w,z).
\end{align*}
Then, the conclusion follows comparing \eqref{DerivativeKernel} and \eqref{ExpansionBergman}.

\end{proof}
\section{Sobolev regularity of the Bergman projection}\label{SobolevBergman}
In this section we prove Theorem \ref{t:Sobolev}. The crucial fact is that the operator $P_{D'_\beta}$ commutes with the differential operators and we can conclude using Theorem \ref{KP2}. 
\begin{proof}
Let us focus for a moment on the term
$$
B_1(w,z)=\frac{\varphi_1(w,z)}{\big[e^{\frac{\pi-i(w_1-\overline{z_1})}{2}}-w_2\overline{z_2}\big]^2}
$$
of the expansion \eqref{ExpansionBergman}. The key observation is the fact that differentiating $B_1(w,z)$ with respect to $w_1$ or $w_2$ is, in some sense, the same as differentiating with respect to $\overline{z_1}$ or $\overline{z_2}$ respectively. In fact, both $D_{w_1}B(w,z)$ and $D_{\overline{z_1}}B(w,z)$ are of the form 
\[
D_{w_1}B_1(w,z)=\frac{\delta_1(w,z)}{\big[e^{\frac{\pi-i(w_1-\overline{z_1})}{2}}-w_2\overline{z_2}\big]^4}
\]
and
\[
D_{\overline{z_1}}B_1(w,z)=\frac{\gamma_1(w,z)}{\big[e^{\frac{\pi-i(w_1-\overline{z_1})}{2}}-w_2\overline{z_2}\big]^4}
\]
where $\delta_1$ and $\gamma_1$ are functions with the same properties of $\varphi_1$. The same is true if we consider higher order derivatives or we differentiate with respect to $w_2$ and $\overline{z_2}$. This property holds for each term of the asymptotic expansion of the kernel $B_{D'_\beta}(w,z)$, therefore, with an abuse of language, we can conclude that, for $k$ fixed integer, $D^{k}_{w}B_{D'_\beta}$ and $D^{k}_{\overline{z}}B_{D'_\beta}$ have the same asymptotic expansion. This expansion is of the form \eqref{ExpansionBergman} and we write $D^{k}_w B_{D'_\beta}\thickapprox D^{k}_{\overline{z}} B_{D'_\beta}$. 

Since $D'_\beta$ is a Lipschitz domain, it satisfies the segment condition. Hence, the set of restriction to $D'_\beta$ of functions in $\clC^{\infty}_0(\dsR^4)$ is dense in $W^{k,p}(D'_\beta)$ (see, e.g., \cite[Theorem $3.22$]{MR2424078}).
Thus, let $f$ be a function  in $\clC^\infty_0(\dsR^4)$ and let $k$ and $l$ be fixed integers.

\begin{align*}
D^{k}_{z_1} D^{l}_{z_2}P_{D'_\beta}f(z_1,z_2)&=D^{k}_{z_1} D^{l}_{z_2} \int_{D'_\beta}f(w)B(z,w)\ dA(w)\\
&=\int_{D'_\beta}f(w)D^{k}_{z_1} D^{l}_{z_2} B(z,w)\ dA(w)\\
&\thickapprox \int_{D'_\beta}f(w)D^{k}_{\overline{w_1}} D^{l}_{\overline{w_2}}B(z,w)\ dA(w)\\
&=-\int_{D'_\beta}\big[D^{k}_{\overline{w_1}} D^{l}_{\overline{w_2}}f(w)\big] B(z,w) dA(w)\\
&=P_{D'_\beta}[D^{k}_{\overline{w_1}} D^{l}_{\overline{w_2}} f](z_1,z_2).
\end{align*}
The proof is easily concluded using this relationship and Theorem \ref{KP2}.
\end{proof}

\bibliographystyle{plain}
\bibliography{bibWormStrip}
\end{document}